\documentclass[11pt]{article}
\usepackage[margin=0.75in]{geometry}

\usepackage{amsthm}
\usepackage{amssymb}
\usepackage{amsmath}
\usepackage{comment}
\usepackage{thm-restate}
\usepackage{url} 
\usepackage{hyperref}
\usepackage[noabbrev,capitalise]{cleveref}

\usepackage[utf8]{inputenc}

\usepackage{color}
\usepackage[normalem]{ulem}

\newcounter{i}
\theoremstyle{plain}
\newtheorem{thm}{Theorem}[section]
\newtheorem{lem}[thm]{Lemma}
\newtheorem{claim}{Claim}[thm]

\newtheorem{cor}[thm]{Corollary}

\newtheorem{conj}[thm]{Conjecture}

{\noindent \emph{Proof.} {}{#1}{}}{\hfill
	$\Diamond$\vspace{1em}}

\theoremstyle{plain} 
\newcommand{\thistheoremname}{}
\newtheorem{genericthm}{\thistheoremname}

\theoremstyle{definition}

\title{The limit in the $(k+2, k)$-Problem of Brown, Erd\H{o}s and S\'os exists for all $k\geq 2$} 

\author{
Michelle Delcourt
\thanks{Department of Mathematics, Toronto Metropolitan University (formerly named Ryerson University),
Toronto, Ontario M5B 2K3, Canada {\tt mdelcourt@torontomu.ca}. Research supported by NSERC under Discovery Grant No. 2019-04269.}
\and
Luke Postle
\thanks{Combinatorics and Optimization Department,
University of Waterloo, Waterloo, Ontario N2L 3G1, Canada {\tt lpostle@uwaterloo.ca}. Partially supported by NSERC
under Discovery Grant No. 2019-04304.}}
\date{\today}

\begin{document}

\maketitle

\begin{abstract} 
Let $f^{(r)}(n;s,k)$ be the maximum number of edges of an $r$-uniform hypergraph on~$n$ vertices not containing a subgraph with $k$~edges and at most $s$~vertices. In 1973, Brown, Erd\H{o}s and S\'os conjectured that the limit $$\lim_{n\to \infty} n^{-2} f^{(3)}(n;k+2,k)$$ exists for all positive integers $k\ge 2$. They proved this for $k=2$. In 2019, Glock proved this for $k=3$ and determined the limit. Quite recently, Glock, Joos, Kim, K\"{u}hn, Lichev and Pikhurko proved this for $k=4$ and determined the limit; we combine their work with a new reduction to fully resolve the conjecture by proving that indeed the limit exists for all positive integers $k\ge 2$.
\end{abstract}

\section{Introduction}

In an $r$-uniform hypergraph, an \emph{$(s,k)$-configuration} is a collection of $k$ edges which span at most $s$ vertices. An $r$-uniform hypergraph is said to be \emph{$(s,k)$-free} if it contains no $(s,k)$-configuration. In 1973, Brown, Erd\H{o}s and S\'os~\cite{BES73} initiated the study of $f^{(r)}(n;s,k)$, the maximum number of edges in an $(s,k)$-free $r$-uniform hypergraph on $n$ vertices, and they made the following conjecture:

\begin{conj}[Brown, Erd\H{o}s and S\'os~\cite{BES73}] \label{conj:BES}
The limit $\lim_{n\to \infty}n^{-2}f^{(3)}(n;k+2,k)$ exists for all $k\ge 2$.
\end{conj}

The main result of this paper is a proof of Conjecture~\ref{conj:BES}.  Indeed, we prove a stronger statement as follows. But first a definition. 

Namely, we introduce the function $g^{(r)}(n;(r-t)k+t,k)$ denoting the maximum number of edges in a $((r-t)k+t,k)$-free $r$-uniform hypergraph on $n$ vertices that is also $((r-t)\ell+t-1,\ell)$-free for all $\ell \in [2,k-1]$.  We note that $f$ is at least $g$. 

Here then is our main result.

\begin{thm}\label{thm:limequality}
For all positive integers $k\ge 2$, $\lim_{n\to \infty} n^{-2}f^{(3)}(n;k+2,k)$ exists and 
$$\lim_{n\to \infty} n^{-2}f^{(3)}(n;k+2,k) = \lim_{n\to \infty} n^{-2}g^{(3)}(n;k+2,k).$$
\end{thm}

As to the history of Conjecture~\ref{conj:BES}, in 1973, Brown, Erd\H{o}s and S\'os~\cite{BES73} noted that Conjecture~\ref{conj:BES} holds for $k=2$ with a limit of $1/6$. The upper bound follows since a $3$-uniform hypergraph has no $(4,2)$-configuration if and only if it is linear (equivalently is a partial Steiner triple system). The lower bound follows from the existence of Steiner triple systems for all $n \equiv 1, 3 ~({\rm mod }~6)$, famously proved by Kirkman~\cite{K47} in 1847. 

In 2019, Glock~\cite{G19} proved that Conjecture~\ref{conj:BES} holds for $k=3$ and determined the limit (namely $1/5$). Quite recently, Glock, Joos, Kim, K\"{u}hn, Lichev and Pikhurko~\cite{GJKKLP22} proved Conjecture~\ref{conj:BES} holds for $k=4$ and determined the limit (namely $7/36$). In light of Theorem~\ref{thm:limequality}, determining the values of the limit in general is of interest. According to~\cite{GJKKLP22}, the previously best known general upper bound to date for $\limsup_{n\to \infty} n^{-2}f^{(3)}(n;k+2,k)$ is $\frac{k-1}{3k}$ by Brown, Erd\H{o}s and S\'os~\cite{BES73} in 1973.

Our second main result is that we improve the general upper bound as follows.

\begin{thm}\label{thm:limupperbound}
Let $k\ge 2$ be an integer. Then
$$\lim_{n\to \infty} n^{-2}f^{(3)}(n;k+2,k) \le \frac{k-1}{4k-2}.$$
\end{thm}

For lower bounds, this problem is related to a famous conjecture of Erd\H{o}s~\cite{E73} from 1973 on the existence of high girth Steiner triple systems. Namely, Erd\H{o}s conjectured that: for every $k$ and large enough $n$ satisfying the necessary divisibility conditions, there exists a Steiner triple system on $n$ vertices with no $(\ell+2,\ell)$-configuration for every $\ell \in [2,k]$. This conjecture of Erd\H{o}s implies a lower bound of $1/6$ on the\\ $\liminf_{n\to \infty} n^{-2}f^{(3)}(n;k+2,k)$. In 1993, Lefmann, Phelps and R\"{o}dl~\cite{LPR93} proved a weak form of Erd\H{o}s' conjecture which implies a lower bound here of $c_k$ for some small positive constant $c_k$ depending on $k$. In the last few years, Erd\H{o}s' conjecture was proved asymptotically by Glock, K\"{u}hn, Lo and Osthus~\cite{GKLO20}, and independently by Bohman and Warnke~\cite{BW19}, which yields the following corollary:

\begin{thm}[\cite{BW19,GKLO20}]\label{thm:sixthlowerbound}
Let $k\ge 2$ be an integer. Then
$$\liminf_{n\to \infty} n^{-2}f^{(3)}(n;k+2,k) \ge \frac{1}{6}.$$
\end{thm}

We remark that Erd\H{o}s' conjecture on high girth Steiner triple systems was proven in full quite recently by Kwan, Sah, Sawhney and Simkin~\cite{KSSS22}. We also remark that $1/6$ is a tight upper bound if the $3$-uniform hypergraph is required to be linear (i.e.~$(4,2)$-free) as in the conjecture of Erd\H{o}s~\cite{E73} above; on the other hand, more efficient lower bound constructions are in theory possible when the $3$-uniform hypergraph is not required to be linear (such as were done for Conjecture~\ref{conj:BES} for the limits when $k=3$ and $k=4$). The issue then that there is no natural conjecture for the value of these limits (indeed, they are not even monotone increasing or decreasing) testifies to the intriguing nature of the problem of determining these limits and even in first proving their existence as we do in Theorem~\ref{thm:limequality}.

\subsection{Higher Uniformities}

Glock, K\"{u}hn, Lo and Osthus~\cite{GKLO20} and Keevash and Long~\cite{KL20} conjectured a generalization of Erd\H{o}s' conjecture: for every $k$ and large enough $n$ satisfying the necessary divisibility conditions, there exists an $(n,r,t)$-Steiner system with no $((r-t)\ell+t,\ell)$-configuration for every $\ell \in [2,k]$. Note for $k=2$ this is the famous Existence of Designs Conjecture, only recently solved by Keevash~\cite{K14} (and more recently using purely combinatorial methods by Glock, K\"{u}hn, Lo and Osthus~\cite{GKLO16}).  So this conjecture is a common generalization of Erd\H{o}s' conjecture and the Existence Conjecture. In light of this conjecture, it is natural to study $n^{-t}f^{(r)}(n;(r-t)k+t,k)$ and indeed in 2020, Shangguan and Tamo~\cite{ST20} conjectured the limit of this exists  for all integers $r > t \ge 2$ and $k\ge 2$.

The conjecture of Glock, K\"{u}hn, Lo and Osthus was proved asymptotically by Delcourt and Postle~\cite{DP22} and independently by Glock, Joos, Kim, K\"{u}hn and Lichev~\cite{GJKKL22}. The asymptotic result in turn provides a lower bound on the limit infimum of $n^{-t}f^{(r)}(n;(r-t)k+t,k)$, generalizing Theorem~\ref{thm:sixthlowerbound}, as follows:

\begin{thm}[\cite{DP22, GJKKL22}]\label{thm:generallowerbound}
Let $k\ge 2$ be an integer. Then
$$\liminf_{n\to \infty} n^{-t}f^{(r)}(n;(r-t)k+t,k) \ge \frac{1}{t! \binom{r}{t}}.$$
\end{thm}

In fact, both groups actually developed a general theory of \emph{hypergraph matchings avoiding forbidden submatchings} (what Glock, Joos, Kim, K\"{u}hn and Lichev called \emph{conflict-free hypergraph matchings}) and provided theorems where almost perfect matchings avoiding a set of forbidden submatchings are found provided various necessary degree and codegree conditions hold. The proofs employ different approaches: Glock, Joos, Kim, K\"{u}hn and Lichev used a random greedy process while we proceeded via the nibble method. 
In either case, this general theory then yields Theorem~\ref{thm:generallowerbound} as an immediate application.

\subsection{Previous Work}

To prove Theorem~\ref{thm:limequality}, we require a result from the recent paper of Glock, Joos, Kim, K\"{u}hn, Lichev and Pikhurko~\cite{GJKKLP22}. Namely, they prove a more general version of Theorem~\ref{thm:generallowerbound} by invoking the general theory of conflict-free hypergraph matchings. Their most general version is useful for proving better lower bounds on the limit infimum of $f$, but we only need the following result:

\begin{thm}[Corollary 3.2 in~\cite{GJKKLP22}]\label{thm:GJKKLP}
Let $k\ge 2$, $r\ge 3$ and $t\in [2,r-1]$ be integers. If $F$ is a $((r-t)k+t,k)$-free $r$-uniform hypergraph which is also $((r-t)\ell+t-1,\ell)$-free for all $\ell\in [2,k-1]$, then 
$$\liminf_{n\to \infty} n^{-t}f^{(r)}(n;(r-t)k+t,k) \ge \frac{e(F)}{v(F)^t}.$$
\end{thm}

We also note that the proof of the main technical result of Glock, Joos, Kim, K\"{u}hn, Lichev and Pikhurko (Theorem 3.1 in~\cite{GJKKLP22}) actually constructs graphs which are not only $((r-t)k+t,k)$-free but are also $((r-t)\ell+t-1,\ell)$-free for all $\ell\in [2,k-1]$. Hence the $f$ in the $\liminf$ in the statement of Theorem 3.1 of~\cite{GJKKLP22} can actually be replaced by the function $g$ we defined. Hence as a corollary one obtains the following slightly stronger version of Theorem~\ref{thm:GJKKLP}.

\begin{thm}\label{thm:GJKKLP3}
Let $k\ge 2$, $r\ge 3$ and $t\in [2,r-1]$ be integers. If $F$ is a $((r-t)k+t,k)$-free $r$-uniform hypergraph which is also $((r-t)\ell+t-1,\ell)$-free for all $\ell\in [2,k-1]$, then
$$\liminf_{n\to \infty} n^{-t}g^{(r)}(n;(r-t)k+t,k) \ge \frac{e(F)}{v(F)^t}.$$
\end{thm}

In this language, the following is an easy corollary of Theorem~\ref{thm:GJKKLP3}:

\begin{cor}\label{cor:GJKKLP2}
For all positive integers $k\ge 2$,
$$\liminf_{n\to \infty} n^{-t}g^{(r)}(n;(r-t)k+t,k) \ge \limsup_{n\to \infty} n^{-t}g^{(r)}(n;(r-t)k+t,k)$$
and hence $\lim_{n\to \infty} n^{-t}g^{(r)}(n;(r-t)k+t,k)$ exists.

\end{cor}

\section{Proof}

Our main technical result is the following then:

\begin{thm}\label{thm:supequality}
For all positive integers $k\ge 2$,
$$\limsup_{n\to \infty} n^{-2}f^{(3)}(n;k+2,k) = \limsup_{n\to \infty} n^{-2}g^{(3)}(n;k+2,k).$$
\end{thm}

Combined with Corollary~\ref{cor:GJKKLP2}, this yields Theorem~\ref{thm:limequality} as an immediate corollary.

In~\cite{GJKKLP22}, an approach to proving Conjecture~\ref{conj:BES} in full was suggested, namely an approach to proving Theorem~\ref{thm:supequality}. The authors asked if one can prove that it is always possible to delete $o(m^2)$ edges from a $(k+2,k)$-free $3$-uniform hypergraph $F$ to make it $(\ell+1,\ell)$-free for all $\ell \in [2,k-1]$? While this strategy is intriguing, it seems difficult in general.

Here we pursue a different approach. We prove (Lemma~\ref{lem:largefreesubgraph} below) that provided that $e(F) \ge \frac{1}{6} \cdot \left(1-\frac{1}{2k}\right) \cdot v(F)^2$, there exists $F'\subseteq F$ with $v(F')\ge \frac{v(F)}{\sqrt{4k}}$ and $\frac{e(F')}{v(F')^2} \ge \frac{e(F)}{v(F)^2}$ that is $(\ell+1,\ell)$-free for all $\ell \in [2,k-1]$. That is, we prove that having $(\ell+1,\ell)$-configurations is `inefficient' for the extremal $(k+2,k)$-free graphs. Given Lemma~\ref{lem:largefreesubgraph}, Theorem~\ref{thm:supequality} follows straightforwardly (but we include a proof for completeness).

To pursue this approach, we first prove a structural lemma. Let $k\ge 2$ be a positive integer and let $\ell \in[2,k-1]$. We say an $(\ell+1,\ell)$-configuration $S$ in a $(k+2,k)$-free $3$-uniform hypergraph $F$ is \emph{$k$-maximal} if there does not exist an $(\ell'+1,\ell')$-configuration $S'$ with $S\subseteq S'$ and $\ell < \ell' \le k-1$.

\begin{lem}\label{lem:maximal}
Let $k\ge 2$ be integer. Let $F$ be a $(k+2,k)$-free $3$-uniform hypergraph. Let $S$ be a $k$-maximal $(\ell+1,\ell)$-configuration of $F$ with $\ell \in [2,k-1]$. Then all of the following hold:
\begin{enumerate}
    \item $e(F[V(S)])\le k-1$,
    \item there does not exist an edge $e\in E(F)$ with $|e\cap V(S)|=2$,
    \item there do not exist edges $e,f \in E(F)$ with $|e\cap V(S)|=|f\cap V(S)|=1$ and $e\setminus V(S) = f \setminus V(S)$, and
    \item the graph $H$ with $V(H)=V(F)\setminus V(S)$ and $E(H) = \{ T \in \binom{V(H)}{2}: \exists e\in E(F),~|e\cap S|=1,~e\setminus V(S)=T\}$ is a forest whose components have size at most $k-\ell$,
\end{enumerate}
and hence
$$e(F) - e(F\setminus V(S)) \le \left(1-\frac{1}{k}\right)\cdot (v(F)-v(S)) + (k-1).$$
\end{lem}
\begin{proof}
First we prove (1). Suppose not. Then $e(F[V(S)])\ge k$ and hence $E(F(V[S]))$ is a $(k,k)$-configuration, contradicting that $F$ is $(k+2,k)$-free. This proves (1).

Next we prove (2). Suppose not. That is, there exists an edge $e\in E(F)$ with $|e\cap V(S)|=2$. If $\ell=k-1$, then $S\cup e$ is a $(k+1,k)$-configuration contradicting that $F$ is $(k+2,k)$-free. So we may assume that $\ell < k-1$. But then $S\cup e$ is an $(\ell+2,\ell+1)$-configuration contradicting that $S$ is $k$-maximal. This proves (2).

Next we prove (3). Suppose not. That is, there exist edges $e,f \in E(F)$ with $|e\cap V(S)|=|f\cap V(S)|=1$ and $e\setminus V(S) = f \setminus V(S)$. If $\ell=k-1$, then $S\cup e$ is a $(k+2,k)$-configuration contradicting that $F$ is $(k+2,k)$-free. So we may assume that $\ell < k-1$. If $\ell=k-2$, then $S\cup \{e,f\}$ is a $(k+1,k)$-configuration contradicting that $F$ is $(k+2,k)$-free. So we may assume that $\ell < k-2$. But then $S\cup \{e,f\}$ is an $(\ell+3,\ell+2)$-configuration contradicting that $S$ is $k$-maximal. This proves (3).

Next we prove (4). Suppose not. That is, $H$ has a component of size greater than $k-\ell$ or $H$ has a cycle of length at most $k-\ell$. First suppose that $H$ has a component of size greater than $k-\ell$. Thus $H$ contains a tree $T$ on exactly $k-\ell+1$ vertices. Let $E(T)=\{e_1,e_2,\ldots, e_{k-\ell}\}$. For each $i\in [k-\ell]$, let $f_i\in E(F)$ such that $f_i\setminus V(S) = e_i$. But then $S\cup \{f_i:~i\in[k-\ell]\}$ is a set of $|S|+(k-\ell)=k$ edges on a set of $|V(S)|+(k-\ell+1) = k+2$ vertices, contradicting that $F$ is $(k+2,k)$-free. 

So we may assume that $H$ has a cycle $C$ of length $j$ where $j\le k-\ell$. Let $E(C)=\{e_1,\ldots, e_j\}$. For each $i\in [j]$, let $f_i\in E(F)$ such that $f_i\setminus V(S) = e_i$. Let $S'=S\cup \{f_i:~i\in[j]\}$. Hence $S'$ is a set of $\ell+j$ edges on a set of $\ell+j+1$ vertices. If $j = k-\ell$, then $S'$ is a $(k+1,k)$-configuration, contradicting that $F$ is $(k+2,k)$-free. So we assume that $j < k-\ell$. But then $S'$ is an $(\ell+j+1,\ell+j)$-configuration where $\ell + j \le k-1$ and $\ell+j > \ell$, contradicting that $S$ is $k$-maximal. This proves (4).

Finally we prove the formula. For each $i\in [3]$, let $E_i := \{ e\in E(F): |e\cap V(S)|=i \}$. Since $F$ is a $3$-uniform hypergraph, we have that 
$$e(F) - e(F\setminus V(S)) = \sum_{i=1}^3 |E_i|.$$
By (1), we have that $|E_3| \le k-1$. By (2), we have that $|E_2| = 0$.  By (3), we have that 
$|E_1|=e(H)$. 
Combining, we find that
$$e(F) - e(F\setminus V(S)) \le k-1 + e(H).$$
Let $a$ be the number of components of $H$. By (4), each component of $H$ has size at most $k-\ell \le k$ and hence 
$$a \ge \frac{1}{k} \cdot v(F\setminus V(S)).$$ 
By (4), $H$ is a forest and hence
\begin{align*} 
e(H) = v(F\setminus V(S)) - a &\le \left(1-\frac{1}{k}\right) \cdot v(F\setminus V(S))\\
&= \left(1-\frac{1}{k}\right) \cdot (v(F)-v(S)). 
\end{align*}
Combining, we find that
\begin{align*}
e(F) - e(F\setminus V(S)) &\le e(H)+k-1 \le \left(1-\frac{1}{k}\right) \cdot (v(F)-v(S)) + (k-1),
\end{align*}
as desired.
\end{proof}

We now apply Lemma~\ref{lem:maximal} to find an induced subgraph that is $(\ell+1,\ell)$-free for all $\ell \in [2,k-1]$ where not too many edges are lost in proportion to the number of vertices deleted.

\begin{lem}\label{lem:diffbound}
Let $k\ge 2$ be integer. Let $F$ be a $(k+2,k)$-free $3$-uniform hypergraph. There exists a $(k+2,k)$-free subgraph $F'$ of $F$ that is also $(\ell+1,\ell)$-free for all $\ell \in [2,k-1]$ such that
$$e(F) - e(F') \le \frac{1}{6} \cdot \left(1-\frac{1}{k}\right) \cdot (v(F)-v(F'))\cdot (v(F)+v(F')+2k).$$
\end{lem}
\begin{proof}
We proceed by induction on $v(F)$. First suppose $F$ is also $(\ell+1,\ell)$-free for all $\ell \in [2,k-1]$. Then $F'=F$ is as desired. So we may assume that there exists an $(\ell+1,\ell)$-configuration in $F$.

Let $S$ be a $k$-maximal $(\ell+1,\ell)$-configuration in $F$ with $\ell \in [2,k-1]$. By Lemma~\ref{lem:maximal}, we have that
$$e(F) - e(F\setminus V(S)) \le \left(1-\frac{1}{k}\right)\cdot (v(F)-v(S)) + (k-1).$$
Note that $v(S)\ge 3$ since $F$ is $3$-uniform hypergraph and hence $v(F\setminus V(S)) < v(F)$. Thus by induction, we find that there exists a $(k+2,k)$-free subgraph $F'$ of $F\setminus V(S)$ that is also $(\ell+1,\ell)$-free for all $\ell \in [2,k-1]$ such that
\begin{align*} e(F\setminus V(S)) - e(F') &\le \frac{1}{6} \cdot \left(1-\frac{1}{k}\right) \cdot (v(F\setminus V(S))-v(F'))\cdot (v(F\setminus V(S))+v(F')+2k)\\
&= \frac{1}{6} \cdot \left(1-\frac{1}{k}\right) \cdot (v(F)-v(S)-v(F'))\cdot (v(F)-v(S)+v(F')+2k)\\
&= \frac{1}{6} \cdot \left(1-\frac{1}{k}\right) \cdot (v(F)-v(F'))\cdot (v(F)+v(F')+2k) \\
&~~~ - \frac{1}{6} \cdot \left(1-\frac{1}{k}\right) \cdot v(S) \cdot (2\cdot v(F) + 2k -v(S)).
\end{align*}
Now using that $2\cdot v(F) - v(S) \ge 2(v(F)-v(S))$, we have that
$$\frac{1}{6} \cdot \left(1-\frac{1}{k}\right) \cdot v(S) \cdot (2\cdot v(F) + 2k -v(S)) \ge \frac{2\cdot v(S)}{6} \cdot \left( \left(1-\frac{1}{k}\right) \cdot \Big(v(F)-v(S)\Big) + (k-1) \right)$$
Using that $v(S) \ge 3$, we find that
$$\frac{1}{6} \cdot \left(1-\frac{1}{k}\right) \cdot v(S) \cdot (2\cdot v(F) + 2k -v(S)) \ge \left(1-\frac{1}{k}\right) \cdot (v(F)-v(S)) + (k-1).$$
Altogether then, we find that
$$e(F) - e(F') \le \frac{1}{6} \cdot \left(1-\frac{1}{k}\right) \cdot (v(F)-v(F'))\cdot (v(F)+v(F')+2k),$$
as desired.
\end{proof}

We need an upper bound on the number of edges in a $(k+2,k)$-free $3$-uniform hypergraph $F$ that is also $(\ell+1,\ell)$-free for all $\ell \in [2,k-1]$. We note that any upper bound on $e(F)$ would suffice for our purposes; for example, $\frac{k}{6}\cdot v(F)^2$ follows easily since $F$ is $(k+2,k)$-free; indeed even the upper bound of $v(F)^3$ which follows from $F$ being $3$-uniform would suffice for our purposes. Nevertheless, finding a better upper bound is also of interest in light of Theorem~\ref{thm:supequality}. Thus we found an improved upper bound which we now prove, but first a definition.

The \emph{$t$-shadow} of an $r$-uniform hypergraph $F$ is the $t$-uniform hypergraph $H$ with $V(H)=V(F)$ and $E(H) = \{ S \in \binom{V(H)}{t}: \exists~e \in E(F) \text{ with } S\subseteq e\}$. If $Y$ is a subset of $E(H)$, then we define the \emph{$t$-shadow of $Y$ in $F$} as the $t$-shadow of the subgraph $F'$ that is the union of the edges in $Y$ (so $V(F')=\bigcup_{e\in Y} V(e)$ and $E(F')=Y$).

Here is the upper bound. Note then that Theorem~\ref{thm:limupperbound} is an immediate corollary of Theorem~\ref{thm:limequality} and the following lemma.

\begin{lem}\label{lem:upperbound}
Let $k\ge 2$ be an integer. Let $F$ be a $(k+2,k)$-free $3$-uniform hypergraph that is also $(\ell+1,\ell)$-free for all $\ell \in [2,k-1]$. Then 
$$e(F) \le \frac{k-1}{2k-1} \cdot \binom{v(F)}{2} \le \frac{k-1}{4k-2} \cdot v(F)^2.$$
\end{lem}
\begin{proof}
Let $H$ be the graph with $V(H)=E(F)$ and 
$$E(H) = \left\{ \{S,T\} \in \binom{V(H)}{2}: |S\cap T| \ge 2\right\}.$$ 
\begin{claim}\label{claim:size}
Let $k\ge 2$ be an integer.  If $T$ is a connected subgraph of $H$, then $v(T)\le k-1$ and $V(T)$ (as a set of edges of $F$) spans exactly $v(T) + 2$ vertices of $F$. 
\end{claim}
\begin{proof}
Let $U$ be any connected subgraph of $H$. Note then that it follows from the definition of $H$ that $V(U)$ (as a set of edges of $F$) spans at most $v(U) + 2$ vertices of $F$ and hence $V(U)$ is a $(v(U)+2,v(U))$-configuration of $F$.

Now first suppose for a contradiction that $v(T)\ge k$. Then there exists a connected subgraph $T'$ of $T$ such that $v(T')=k$. But then from above, we find that $V(T')$ is a $(k+2,k)$-configuration of $F$, contradicting that $F$ is $(k+2,k)$-free. Thus we find that $v(T)\le k-1$. 

Furthermore, if $V(T)$ spans at most $v(T)+1$ vertices of $F$, then $V(T)$ is a $(v(T)+1,v(T))$-configuration of $F$, contradicting that $F$ is $(\ell+1,\ell)$-free for all $\ell \in [1,k-1]$ (where the fact that $F$ is $(2,1)$-free follows immediately from the fact that $F$ is a $3$-uniform hypergraph). Hence from above, we find that $V(T)$ spans exactly $v(T)+2$ vertices of $F$ as desired.
\end{proof}

\begin{claim}\label{claim:shadow}
If $T$ is a connected subgraph of $H$, then the $2$-shadow of $V(T)$ in $F$ has exactly $2\cdot v(T)+1$ edges.
\end{claim}
\begin{proof}
We proceed by induction on $v(T)$. Let $a$ be the number of edges in the $2$-shadow of $V(T)$ in $F$. If $v(T)=1$, then $a=3$ as desired. 

So we may assume that $v(T)\ge 2$. Let $X_T$ be the set of vertices of $F$ that $V(T)$ spans. By Claim~\ref{claim:size}, $|X_T| = v(T) + 2$. Since $T$ is connected, there exists a connected subgraph $T'$ of $T$ such that $v(T')=v(T)-1$.  Let $V(T)\setminus V(T') = \{e\}$.  Let $X_{T'}$ be the set of vertices of $F$ that $V(T')$ spans. By Claim~\ref{claim:size}, $|X_{T'}| = v(T') + 2$. 

Note that $X_{T'} \subseteq X_T$. Let $X_T\setminus X_{T'} = \{v\}$. It follows that $v\in e$ and yet $v\not\in f$ for all $f\in V(T')$. Let $e=vv_1v_2$. Hence $vv_1,vv_2$ are in the $2$-shadow of $V(T)$ in $F$ but are not in the $2$-shadow of $V(T')$ in $F$. By definition of $H$, it then follows that $v_1v_2$ is in the $2$-shadow of $V(T')$.

By induction, the number of edges in the $2$-shadow of $V(T')$ is exactly $2\cdot v(T')+ 1$. Thus the number of edges in the $2$-shadow of $V(T)$ is exactly $(2\cdot v(T')+ 1)+2 = 2\cdot v(T) + 1$, as desired. 
\end{proof}

Let $T_1,\ldots, T_r$ be the components of $H$. For each $i\in [r]$, let $F_i$ be the $2$-shadow of $V(T_i)$ in $F$. By Claim~\ref{claim:shadow}, we have for each $i\in [r]$ that $e(F_i)=2\cdot v(T_i)+1$. By the definition of $H$, it follows that $E(F_i)\cap E(F_j) = \emptyset$ for each $i\ne j \in [r]$. By Claim~\ref{claim:size}, $v(T_i)\le k-1$ for all $i\in [r]$.  It follows that for each $i\in [r]$, we have that
$$v(T_i) = \frac{v(T_i)}{2\cdot v(T_i)+1} \cdot e(F_i) \le \frac{k-1}{2k-1} \cdot e(F_i).$$
Since the $F_i$ are edge-disjoint, we find that
$$\binom{v(F)}{2} \ge \sum_{i=1}^r e(F_i) \ge \frac{2k-1}{k-1} \sum_{i=1} v(T_i) = \frac{2k-1}{k-1}\cdot e(F).$$
Rearranging, we find that
$$e(F) \le \frac{k-1}{2k-1} \cdot \binom{v(F)}{2} \le \frac{k-1}{4k-2} \cdot v(F)^2,$$
as desired.
\end{proof}

We are now prepared to state and prove Lemma~\ref{lem:largefreesubgraph} about finding a large $(\ell+1,\ell)$-free subgraph for all $\ell \in [2, \ell-1]$ that is as dense as the original $(k+2,k)$-free $3$-uniform hypergraph.

\begin{lem}\label{lem:largefreesubgraph} Let $k\ge 2$ be an integer. Let $F$ be a $(k+2,k)$-free $3$-uniform hypergraph with $v(F) \ge 8k^2$ and $e(F) \ge \frac{1}{6}\left(1-\frac{1}{2k}\right) \cdot v(F)^2$. Then there exists a $(k+2,k)$-free subgraph $F'$ of $F$ that is also $(\ell+1,\ell)$-free for all $\ell \in [2,k-1]$ such that both of the following hold:
\begin{enumerate}
    \item $v(F') \ge  \frac{v(F)}{\sqrt{4k}}$, and
    \item $\frac{e(F')}{v(F')^2} \ge \frac{e(F)}{v(F)^2}$.
\end{enumerate}
\end{lem}
\begin{proof}
By Lemma~\ref{lem:diffbound}, there exists a $(k+2,k)$-free subgraph $F'$ of $F$ that is also $(\ell+1,\ell)$-free for all $\ell \in [2,k-1]$ such that
$$e(F) - e(F') \le \frac{1}{6} \cdot \left(1-\frac{1}{k}\right) \cdot (v(F)-v(F'))\cdot (v(F)+v(F')+2k).$$ 
By Lemma~\ref{lem:upperbound}, we have that
$$e(F') \le \frac{k-1}{4k-2} \cdot v(F')^2 \le \frac{1}{4} \cdot v(F')^2.$$
Combining, we find that
$$e(F) \le  \frac{1}{4} \cdot v(F')^2 + \frac{1}{6} \cdot \left(1-\frac{1}{k}\right) \cdot (v(F)-v(F'))\cdot (v(F)+v(F')+2k).$$
Since $e(F) \ge \frac{1}{6}\left(1-\frac{1}{2k}\right) \cdot v(F)^2$, we have that as $k \geq 2$,
$$\frac{1}{12k} \cdot v(F)^2 \le \frac{1}{6} \cdot v(F')^2 + \frac{k-1}{3} \cdot (v(F)-v(F')).$$ 
Multiplying by $12k$ and upper bounding $(k-1)(v(F)-v(F')$ by $k\cdot v(F)$ yields
$$v(F)^2 \le 2k\cdot v(F')^2 + 4k^2 \cdot v(F) \le 2k\cdot v(F')^2 + \frac{1}{2} \cdot v(F)^2,$$
where for the last inequality we used that $v(F) \ge 8k^2$. Rearranging gives
$$\frac{1}{2} v(F)^2 \le 2k\cdot  v(F')^2,$$
Taking square roots, it now follows that
$$\frac{v(F)}{\sqrt{4k}} \le v(F').$$
This proves (1).

Now we prove (2). Let $a = \frac{e(F)}{v(F)^2}$. Suppose towards a contradiction that $\frac{e(F')}{v(F')^2} < a$. Note this implies that $v(F') < v(F)$. Moreover, it implies that
$$e(F)-e(F') \ge a \cdot \left(v(F)^2-v(F')^2\right) = a \cdot (v(F)-v(F'))\cdot (v(F)+v(F')).$$
But then using that $v(F)-v(F') > 0$, it follows that
$$a\cdot (v(F)+v(F')) \le \frac{e(F)-e(F')}{v(F)-v(F')} \le \frac{1}{6} \cdot \left(1-\frac{1}{k}\right) \cdot  (v(F)+v(F')+2k).$$
Using that $a\ge \frac{1}{6} \left(1-\frac{1}{2k}\right)$, we find that
$$\frac{1}{12k} \cdot (v(F)+v(F')) \le \frac{k-1}{3}.$$
But this implies that
$$v(F) \le 4k^2,$$
contradicting that $v(F) \ge 8k^2$. This proves (2).
\end{proof}

We now prove Theorem~\ref{thm:supequality} which stated that the $\limsup n^{-2}f^{(3)}(n;k+2,k)$ and\\ $\limsup n^{-2}g^{(3)}(n;k+2,k)$ were equal.

\begin{proof}[Proof of Theorem~\ref{thm:supequality}]
Suppose not. Let $a := \limsup_{n\to \infty} n^{-2}f^{(3)}(n;k+2,k)$ and \\ $b := \limsup_{n\to \infty} n^{-2}g^{(3)}(n;k+2,k)$. Thus, we have that $a\ne b$. Let $\varepsilon := a - b$. Since $f^{(3)}(n;k+2,k) \ge g^{(3)}(n;k+2,k)$ for all positive integers $n$, we find that $a \ge b$. Since $a\ne b$, we have that $\varepsilon > 0$.  

Let $c := \max\left\{a- \frac{\varepsilon}{2}, \frac{1}{6}\cdot\left(1-\frac{1}{2k}\right)\right\}$. Since $a\ge \frac{1}{6}$ by Theorem~\ref{thm:sixthlowerbound}, we find that $c < a$. Yet by definition, $c \ge a - \frac{\varepsilon}{2} > b$. Let 
$$I_f := \{ n \in \mathbb{N}: n^{-2}f^{(3)}(n;k+2,k) \ge c \}.$$ 
Let
$$I_g := \{ n \in \mathbb{N}: n^{-2}g^{(3)}(n;k+2,k) \ge c \}.$$ 
Since $c < a$, $|I_f|$ is infinite. Since $c > b$, $|I_g|$ is finite. Let 
$n_0 := \max \{n: n\in I_g\}$. Let 
$$n_0^{\prime} := \max \left\{\sqrt{4k}\cdot n_0 + 1, 8k^2 \right\}.$$ 
Since $|I_f|$ is infinite, there exists $n_1 \in I_f$ such that $n_1\ge n_0^{\prime}$. That is, there exists a $(k+2,k)$-free $3$-uniform hypergraph $F$ on $n_1$ vertices such that $e(F) \ge c \cdot v(F)^2$. 

Note that $v(F) \ge 8k^2$ by definition of $n_0^{\prime}$. Since $c \ge \frac{1}{6}\cdot\left(1-\frac{1}{2k}\right)$, we have that $e(F) \ge \frac{1}{6}\cdot\left(1-\frac{1}{2k}\right) \cdot v(F)^2$. Hence by Lemma~\ref{lem:largefreesubgraph}, there exists a $(k+2,k)$-free subgraph $F'$ of $F$ that is also $(\ell+1,\ell)$-free for all $\ell \in [2,k-1]$ such that both of the following hold:
\begin{enumerate}
    \item $v(F') \ge  \frac{v(F)}{\sqrt{4k}}$, and
    \item $\frac{e(F')}{v(F')^2} \ge \frac{e(F)}{v(F)^2}$.
\end{enumerate}
By (1) and the definition of $n_0^{\prime}$, we find that $v(F') > n_0$. By (2), we find that $e(F') \ge c \cdot v(F')^2$. Hence $v(F') \in I_g$, contradicting that $n_0$ is the maximum value in $I_g$. 
\end{proof}

\section{Concluding Remarks}

First we note that our proof of Theorem~\ref{thm:supequality} and hence of Theorem~\ref{thm:limequality} works for multi-hypergraphs, that is hypergraphs where the edges form a multiset instead of a set. The multi-hypergraph case is in fact harder as one needs to forbid $(3,2)$-configurations (which do not exist if the edges form a set). Indeed, this added complication is why we took care to ensure the components of $H$ in Lemma~\ref{lem:maximal} had bounded size; the extra factor of $\left(1-\frac{1}{k}\right)$ resulting from the components having bounded size would not have been needed if we assumed $v(S)\ge 4$ in the proof of Lemma~\ref{lem:diffbound} - which we could have done if we did not need to forbid $(3,2)$-configurations. 

Second we note that the proof of Theorem~\ref{thm:supequality} easily generalizes to the case when $t=2$ to yield the following theorem. 

\begin{thm}\label{thm:supequality2}
For all positive integers $k\ge 2$ and $r\ge 3$,
$$\limsup_{n\to \infty} n^{-2}f^{(r)}(n;(r-2)k+2,k) = \limsup_{n\to \infty} n^{-2}g^{(r)}(n;(r-2)k+2,k).$$
\end{thm}

Combined with Corollary~\ref{cor:GJKKLP2}, this yields the following.

\begin{thm}\label{thm:limitexists2}
For all positive integers $k\ge 2$ and $r\ge 3$, $\lim_{n\to \infty} n^{-2}f^{(r)}(n;(r-2)k+2,k)$ exists and 
$$\lim_{n\to \infty} n^{-2}f^{(r)}(n;(r-2)k+2,k) = \limsup_{n\to \infty} n^{-2}g^{(r)}(n;(r-2)k+2,k).$$
\end{thm}

This proves the conjecture of Shangguan and Tamo~\cite{ST20} for the case $t=2$. While our proof easily modifies to yield these theorems as mentioned above, we opted to write the proof only for the case $r=3$ since only this case was originally conjectured by Brown, Erd\H{o}s and S\'os and we did not want to obscure the proof with added generality. 

The needed changes to our proof to work for the general case $t=2$ are mostly superficial. The main substantive change is that in outcome 4 of Lemma~\ref{lem:maximal}, we instead obtain an $(r-1)$-uniform hypergraph $H$ that is a hyperforest with bounded component size. Here though the bounded component size is unneeded even when considering multihypergraphs since the acceptable amount of loss in Lemma~\ref{lem:diffbound} is (by taking derivatives) $\frac{2r}{r(r-1)} = \frac{2}{r-1}$ while the upper bound in Lemma~\ref{lem:maximal} becomes $\left(1-\frac{1}{k}\right) \frac{1}{r-2}$. Thus the extra factor of $\left(1-\frac{1}{k}\right)$ is unneeded when $r\ge 4$ since in that case $\frac{2}{r-1} > \frac{1}{r-2}$. We omit the details.

Shangguan~\cite{s22} has also independently generalized our proof to the case $t=2$ (albeit with a slightly longer, slightly more complicated proof). Nevertheless, it seems new ideas are needed to prove the limits exist when $t\ge 3$.

\bibliographystyle{plain}
\bibliography{mdelcourt}

\end{document}